\documentclass[11pt]{amsart}  
\usepackage{setspace}
\newcommand{\R}{\mathbb{R}}
\newcommand{\E}{\mathbb{E}}
\def\ceil#1{\lceil #1 \rceil}

\calclayout

\usepackage{mathrsfs,mathtools}
\usepackage{xparse}
\usepackage{enumerate}
\usepackage{graphicx}
\usepackage{float}
\usepackage{xcolor}

\usepackage{tikz}
\usetikzlibrary{calc}

\usepackage{aliascnt}

\usepackage[colorlinks=true,allcolors=blue,backref=page]{hyperref}  
\usepackage[noabbrev,capitalize,nameinlink]{cleveref}

\definecolor{darkred}{RGB}{200,0,0}
\hypersetup{colorlinks={true},linkcolor={darkred},citecolor={darkred}}

\makeatletter
\newcommand{\thmheadlinktarget}{}
\newcommand{\linktheoremto}[1]{\gdef\thmheadlinktarget{#1}}
\newcommand{\unlinktheorem}{\gdef\thmheadlinktarget{}}

\newtheoremstyle{linkedplain}
  {6pt plus 2pt minus 1pt}   
  {6pt plus 2pt minus 1pt}   
  {\itshape}                 
  {}                         
  {\bfseries}                
  {.}                        
  {.5em}                     
  {%
    \if\relax\detokenize\expandafter{\thmheadlinktarget}\relax
      \thmname{#1}\thmnumber{ #2}%
    \else
      \hyperlink{\thmheadlinktarget}{\thmname{#1}\thmnumber{ #2}}%
    \fi
    \thmnote{ (#3)}%
  }
\makeatother

\theoremstyle{linkedplain}

\newtheorem{theorem}{Theorem}
\newtheorem*{theorem*}{Theorem}

\newaliascnt{lemma}{theorem}
\newtheorem{lemma}[lemma]{Lemma}
\aliascntresetthe{lemma}
\newtheorem*{lemma*}{Lemma}

\newaliascnt{claim}{theorem}
\newtheorem{claim}[claim]{Claim}
\aliascntresetthe{claim}
\newtheorem*{claim*}{Claim}

\newaliascnt{observation}{theorem}

\aliascntresetthe{observation}
\newtheorem*{observation*}{Observation}

\newaliascnt{prop}{theorem}
\newtheorem{prop}[prop]{Proposition}
\aliascntresetthe{prop}

\newaliascnt{corollary}{theorem}
\newtheorem{corollary}[corollary]{Corollary}
\aliascntresetthe{corollary}

\newaliascnt{example}{theorem}

\aliascntresetthe{example}

\theoremstyle{definition}
\newtheorem*{defn}{Definition}

\theoremstyle{remark}
\newaliascnt{remark}{theorem}

\aliascntresetthe{remark}
\newtheorem*{remark*}{Remark}

\makeatletter
\newcommand{\restatetheoremname}{}

\theoremstyle{linkedplain}
\newtheorem*{innerrestatetheorem}{\restatetheoremname}

\NewDocumentEnvironment{restatetheorem}{m o}
  {%
    \begingroup
    \renewcommand{\restatetheoremname}{%
      \hypertarget{restatement:#1}{}%
      Theorem~\ref*{#1}%
      \IfValueT{#2}{\ (\hyperref[#1]{#2})}%
    }%
    \begin{innerrestatetheorem}%
    \ignorespaces
  }
  {%
    \end{innerrestatetheorem}%
    \endgroup
    \ignorespacesafterend
  }
\makeatother

\newcommand{\cK}{\mathcal{K}}

\renewcommand{\S}{\mathbb{S}}
\newcommand{\relu}{\mathsf{ReLU}}

\newcommand{\cH}{\mathcal{H}}
\newcommand{\cP}{\mathcal{P}}

\newcommand{\eps}{\varepsilon}
\newcommand{\ip}[2]{\langle #1,#2 \rangle}
\DeclareMathOperator{\conv}{\mathsf{conv}}

\newcommand{\N}{\mathbb N}

\numberwithin{equation}{section}

\crefname{theorem}{theorem}{theorems}
\Crefname{theorem}{Theorem}{Theorems}

\crefname{lemma}{lemma}{lemmas}
\Crefname{lemma}{Lemma}{Lemmas}

\crefname{claim}{claim}{claims}
\Crefname{claim}{Claim}{Claims}

\crefname{observation}{observation}{observations}
\Crefname{observation}{Observation}{Observations}

\crefname{prop}{proposition}{propositions}
\Crefname{prop}{Proposition}{Propositions}

\crefname{corollary}{corollary}{corollaries}
\Crefname{corollary}{Corollary}{Corollaries}

\crefname{remark}{remark}{remarks}
\Crefname{remark}{Remark}{Remarks}

\crefname{example}{example}{examples}
\Crefname{example}{Example}{Examples}

\begin{document}

\title
{A simplex-based measure of symmetry}

\author[Bakaev]{Egor Bakaev}
\address{Department of Computer Science, University of Copenhagen}

\author[Yehudayoff]{Amir Yehudayoff}
\address{Department of Computer Science, University of Copenhagen,
and Department of Mathematics, Technion-IIT}

\begin{abstract}
For compact convex sets $L,K \subset \R^n$,
denote by $\lambda_K(L)$ the smallest size of a homothet of $K$ that contains $L$.
We define a measure of symmetry based on the $n$-simplex $\Delta = \Delta^n \subset \R^n$ as the ratio
\[
\rho_\Delta(L):=\frac{\lambda_{-\Delta}(L)}{\lambda_{\Delta}(L)}.
\]
We study this measure and deduce the following results:

(1) The classical Minkowski measure of symmetry
$m^*(L)$ can be defined as an affine-invariant version of $\rho_\Delta(L)$. 

(2) We improve the stability analysis for the Minkowski measure of symmetry;
if  $m^*(L)\ge n-\eps$
then $L$ is $\tfrac{1}{1-\eps}$-close to $\Delta$
in the
Banach--Mazur distance.

(3) We obtain a novel characterization of simplices as the only 
convex bodies~$K$ for which the function $L \mapsto \lambda_K(L)$ is additive
(a property we term ``outer additivity'').

(4) Motivated by the expressivity of $\relu$ neural networks,
we study the depth complexity of polytopes in $\R^n$ under the two operations: Minkowski sum and convex hull of a union. 
We prove the sharp bound $\rho_\Delta(P) \leq 2^d -1$ for every polytope~$P$ of depth complexity $d$. 
In other words, simplices cannot be approximated by low-depth polytopes. 

\end{abstract}

\maketitle

\tableofcontents

\newpage

\section{Introduction}
\label{sec:introduction}

\subsection{Measures of symmetry}
A set $K \subset \R^n$ is centrally symmetric if there is $o \in \R^n$ such that $x \in K-o \iff -x \in K-o$.
Measures of symmetry capture how close a given convex set is to being
centrally symmetric, and are (assumed to be) affine invariant. 
The theory of measures of symmetry for convex sets was initiated by
Gr{\"u}nbaum \cite{grunbaum1963measures}. The recent book by Toth
\cite{toth2015measures} gives a comprehensive account of this topic. The standard measure of symmetry is the Minkowski measure which is defined as follows.

\begin{defn}
For a convex body $K \subset \R^n$ and $o\in\operatorname{int}(K)$, the Minkowski symmetry of $K$ at $o \in \R^n$ is
\[
m_K(o):=\inf\{\alpha\ge 1 : -(K-o)\subseteq \alpha (K-o)\}.
\]
The \emph{Minkowski measure (of symmetry)} of $K$ is
\[
m^*(K):=\inf_{o\in \operatorname{int}(K)} m_K(o).
\]
\end{defn}

The infimum defining $m^*(K)$ is attained, and any minimizer is called a \emph{Minkowski center}.

The following theorem gives an equivalent description of $m_K(o)$ in terms of
slabs containing $K$. The support function of $K$ is denoted by $h_K$, and the sphere is denoted by $\S^{n-1} \subset \R^n$.

\begin{theorem}[Section 3.2 in \cite{toth2015measures}]
\label{thm:minkowski-slabs}
For a convex body $K \subset \R^n$ and $o\in\operatorname{int}(K)$, the Minkowski symmetry of $K$ at $o$ is
\[
m_K(o)=\max_{u\in\S^{n-1}}\frac{h_{K-o}(-u)}{h_{K-o}(u)}.
\]
\end{theorem}

We define a non-affine-invariant measure (of symmetry) based on simplices.

\subsection{The simplex-based measure} 
\label{subsection:simplex-based-measure}

A homothet of $K \subset \R^n$
is a set of the form $H=x+\lambda K$ where $x \in \R^n$ and $\lambda  \geq 0$ is called the coefficient of homothety (or size in short). 
For a compact convex set $K\subset \mathbb R^n$, the family of (nonnegative) homothets of $K$ is
\[
\mathcal H_K=\{x+\lambda K:x\in\mathbb R^n,\ \lambda\ge 0\} .
\]
For a compact convex $L\subset\mathbb R^n$, define
\[
\lambda_K(L):=\inf\{\lambda\ge 0:\exists x\in\mathbb R^n
\text{ such that } L\subset x+\lambda K\}.
\]
In words, it is the smallest size of a homothet of $K$ containing $L$.
If $L \subseteq \R^n$ is not a single point
and the interior of $K$ is non-empty then $0<\lambda_{K}(L)<\infty$.

We write $\Delta = \Delta^n$ for the standard regular simplex in $\R^n$ obtained by orthogonally projecting $\conv\{e_1,\ldots,e_{n+1}\}\subset\R^{n+1}$ to the subspace orthogonal to $(1,1,\ldots,1)$, where $e_1,\ldots,e_{n+1}$ are the standard unit vectors.

\begin{figure}
\centering
\begin{tikzpicture}[
  line cap=round,
  line join=round,
  edge/.style={black, line width=1.6pt},
  lab/.style={black, font=\bfseries\Large}
]

\newcommand{\EqTri}[3]{
  \coordinate (#1A) at #2;
  \coordinate (#1B) at ($(#1A)+(#3,0)$);
  \coordinate (#1C) at ($(#1A)+({#3/2},{#3*sqrt(3)/2})$);
  \draw[edge] (#1A)--(#1B)--(#1C)--cycle;
}

\node at (-7.7,1.8) {$\Delta:$};
\EqTri{K}{(-7.3,1.3)}{1.1}

\pgfmathsetmacro{\a}{2.8}
\pgfmathsetmacro{\b}{4.2}
\pgfmathsetmacro{\t}{3.7}
\pgfmathsetmacro{\hu}{sqrt(3)*\a}
\pgfmathsetmacro{\hd}{sqrt(3)*\b}

\coordinate (U1) at (-\a,0);
\coordinate (U2) at ( \a,0);
\coordinate (U3) at (0,\hu);

\coordinate (D1) at (-\b,\t);
\coordinate (D2) at ( \b,\t);
\coordinate (D3) at (0,\t-\hd);

\draw[edge] (U1)--(U2)--(U3)--cycle;
\draw[edge] (D1)--(D2)--(D3)--cycle;

\coordinate (H1) at ({\t/sqrt(3)-\a},\t);
\coordinate (H2) at ({\a-\t/sqrt(3)},\t);
\coordinate (H3) at ({(\a+\b)/2-\t/(2*sqrt(3))},{(\a-\b)*sqrt(3)/2+\t/2});
\coordinate (H4) at ({\b-\t/sqrt(3)},0);
\coordinate (H5) at ({\t/sqrt(3)-\b},0);
\coordinate (H6) at ({-(\a+\b)/2+\t/(2*sqrt(3))},{(\a-\b)*sqrt(3)/2+\t/2});

\coordinate (P1) at ($(H1)!0.50!(H2)$);
\coordinate (P2) at ($(H2)!0.42!(H3)$);
\coordinate (P3) at ($(H3)!0.40!(H4)$);
\coordinate (P4) at ($(H4)!0.28!(H5)$);
\coordinate (P5) at ($(H4)!0.88!(H5)$);
\coordinate (P6) at ($(H5)!0.40!(H6)$);
\coordinate (P7) at ($(H6)!0.42!(H1)$);

\filldraw[fill=gray!35, draw=black, line width=0.9pt]
  (P1)--(P2)--(P3)--(P4)--(P5)--(P6)--(P7)--cycle;
\node[lab, fill=white, inner sep=1.5pt] at (0,1.55) {$L$};

\end{tikzpicture}
\caption{An illustration of $\rho_\Delta(L)$.}
\label{fig:rho}
\end{figure}
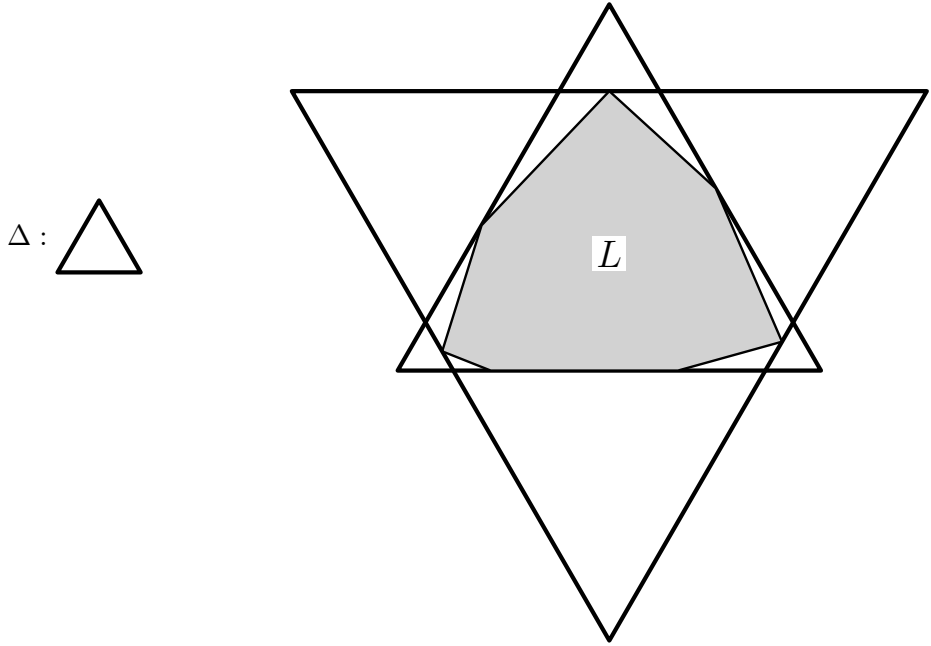

\begin{defn}
For a compact convex set $L \subset \R^n$ that is not a single point,
the \emph{simplex-based measure} is
\[
\rho_\Delta(L):=\frac{\lambda_{-\Delta}(L)}{\lambda_{\Delta}(L)}.
\]
By convention, if $L$ is a point, then $\rho_\Delta(L):=0$. 
\end{defn}

\Cref{fig:rho} illustrates the definition. We use the term {\em measure} because it resembles a {\em measure of symmetry} but it is not a measure of symmetry in the common sense. A~measure of symmetry  should be maximal if and only if the body is symmetric~\cite{toth2015measures}, and $\rho$ does not satisfy this.

Let us start with describing some basic properties of this measure. 
It is invariant under scaling: $
\rho_\Delta(L)=\rho_\Delta(\lambda L)
$ for all $\lambda>0$.
It can be bounded by
\[\frac{1}{n} \le \rho_\Delta(L)\le n.\]
If $L \subset \R^n$ is not a point and is 
centrally symmetric then
$\rho_\Delta(L)=1$. The converse is false: $\rho_\Delta(L)=1$ can also hold for non-symmetric bodies --- one can take a ball and remove a small part of it from a region where it does not touch the boundaries of the homothets of $\Delta$ and $-\Delta$ in which it is inscribed.
Evaluating $\rho_\Delta(K)$ only requires $2(n+1)$ support-function values, so for polytopes this quantity is, in many cases, easy to compute.

We show that the Minkowski measure can also be described in terms
of containment in simplices. \Cref{sec:minkowski-equivalence}
shows that the supremum of $\rho_\Delta(K)$ over all affine images of $\Delta$ is equal to
$m^*(K)$.
In words, 
an affine-invariant version of $\rho_\Delta$ is $m^*(K)$.
This can be viewed as a refinement of \Cref{thm:minkowski-slabs}. Instead of inscribing
$K$ into all possible slabs, we inscribe it into all possible simplices. Consequently, properties of $\rho_\Delta$ yield corresponding properties of the Minkowski measure.

\subsection{Results}
\label{subsec:results}

\subsubsection*{An equivalent definition of the Minkowski measure (\Cref{sec:minkowski-equivalence})}

The following is an affine invariant version of the simplex-based measure.
\begin{defn}
For a non-singleton compact convex set $K \subset \R^n$, let $\mathscr S(K)$ denote the family of nondegenerate simplices in $E_K$ whose barycenter is $0$, and define
\[
\alpha(K):=\sup_{\Delta\in\mathscr S(K)}\rho_\Delta(K).
\]
\end{defn}

\linktheoremto{restatement:thm:main-equality}
\begin{theorem}
\label{thm:main-equality}
For every non-singleton compact convex set $K \subset \R^n$,
\[
\alpha(K)=m^*(K).
\]
\end{theorem}
\unlinktheorem

\subsubsection*{Stability of Minkowski measure (\Cref{sec:stability})}

For a convex body $L$ and a simplex $K$, let $\mu_K(L)$ be the size of the largest homothet of $K$ contained in $L$, and define 
\[
d_K(L):=\frac{\lambda_K(L)}{\mu_K(L)}.
\]
See \Cref{fig:d} for an illustration.
In words, it is the smallest ratio between the sizes of a simplex outside $L$ and a simplex inside $L$,
which is similar to the
condition in the Banach--Mazur distance.
It always holds that 
$d_K(L) \geq 1$ and $d_K(L)=1$
iff $L$ is a nonnegative homothet of $K$. 
The main comparison between the two measures is the following pair of inequalities:

\linktheoremto{restatement:thm:banach}
\begin{theorem}
\label{thm:banach}
For the simplex $\Delta \subset \R^n$ 
and a convex body $L \subset \R^n$,
\[
\frac{n}{d_\Delta(L)}
\le
\rho_\Delta(L)
\le
n-1+\frac{1}{d_\Delta(L)}.
\]
\end{theorem}
\unlinktheorem

Interestingly, the proof 
is based on a Kadets-type result proved by Akopyan and Karasev~\cite{akopyan2012kadets}.

\begin{figure}
\centering
\begin{tikzpicture}[
  line cap=round,
  line join=round,
  edge/.style={black, line width=1.6pt},
  innertri/.style={black, line width=1.4pt},
  lab/.style={black, font=\bfseries\Large}
]

\newcommand{\RefTriangle}[2]{
  \coordinate (KA) at #1;
  \coordinate (KB) at ($(KA)+(#2,0)$);
  \coordinate (KC) at ($(KA)+({#2/2},{#2*sqrt(3)/2})$);
  \draw[edge] (KA)--(KB)--(KC)--cycle;
}

\node at (-7.7,1.8) {$K:$};
\RefTriangle{(-7.3,1.3)}{1.1}

\pgfmathsetmacro{\a}{2.8}
\pgfmathsetmacro{\b}{4.2}
\pgfmathsetmacro{\t}{3.7}
\pgfmathsetmacro{\hu}{sqrt(3)*\a}

\coordinate (U1) at (-\a,0);
\coordinate (U2) at ( \a,0);
\coordinate (U3) at (0,\hu);
\draw[edge] (U1)--(U2)--(U3)--cycle;

\coordinate (H1) at ({\t/sqrt(3)-\a},\t);
\coordinate (H2) at ({\a-\t/sqrt(3)},\t);
\coordinate (H3) at ({(\a+\b)/2-\t/(2*sqrt(3))},{(\a-\b)*sqrt(3)/2+\t/2});
\coordinate (H4) at ({\b-\t/sqrt(3)},0);
\coordinate (H5) at ({\t/sqrt(3)-\b},0);
\coordinate (H6) at ({-(\a+\b)/2+\t/(2*sqrt(3))},{(\a-\b)*sqrt(3)/2+\t/2});

\coordinate (P1) at ($(H1)!0.50!(H2)$);
\coordinate (P2) at ($(H2)!0.42!(H3)$);
\coordinate (P3) at ($(H3)!0.40!(H4)$);
\coordinate (P4) at ($(H4)!0.28!(H5)$);
\coordinate (P5) at ($(H4)!0.88!(H5)$);
\coordinate (P6) at ($(H5)!0.40!(H6)$);
\coordinate (P7) at ($(H6)!0.42!(H1)$);

\filldraw[fill=gray!35, draw=black, line width=0.9pt]
  (P1)--(P2)--(P3)--(P4)--(P5)--(P6)--(P7)--cycle;
\node[lab, fill=white, inner sep=1.5pt] at (0,1.55) {$L$};

\coordinate (T1) at (-1.966399152, 0.294096760);
\coordinate (T2) at ( 1.966399152, 0.294096760);
\coordinate (T3) at ( 0.000000000, 3.700000000);
\draw[innertri] (T1)--(T2)--(T3)--cycle;

\end{tikzpicture}
\caption{An illustration of $d_K(L)$.}
\label{fig:d}
\end{figure}
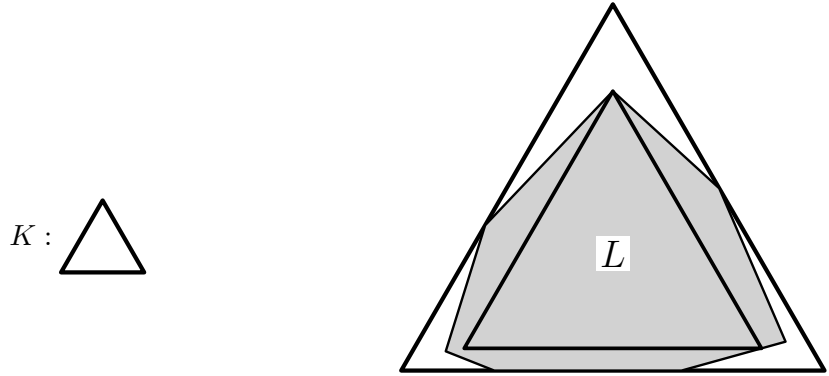

As an application of our techniques, we can bound the stability of the Minkowski measure of symmetry. 
The body in $\R^n$ with the largest Minkowski measure $m^*$ is the simplex and $m^*(\Delta) = n$.
Stability refers to a statement of the form: if $m^*(L)$ is close to the maximum $n$ then $L$ is close to $\Delta$
in the Banach--Mazur distance
$d_{\rm BM}$.
Stability estimates of this kind were  obtained in
\cite{boroczky1996around, boroczky2005stability, guo2005stability,toth2013notes}, and
to the best of our knowledge the strongest previously known result is due to Schneider (Theorem 2.1 in \cite{schneider2009stability}, see also \cite{toth2015measures}, Theorem 3.2.4). 
\begin{theorem}[\cite{schneider2009stability}]
If  $L\subset \R^n$ is a convex body and 
$m^*(L)\ge n-\eps$
for some $0\leq \eps<\frac{1}{n}$,
then
\[
 d_{\rm BM}(L,\Delta) < 1 +  \frac{(n+1)\eps}{1-n\eps}.
\]
\end{theorem}

We prove the following stronger bound.

\linktheoremto{restatement:thm:stabBM}
\begin{theorem}
\label{thm:stabBM}
If  $L\subset \R^n$ is a convex body and 
$m^*(L)\ge n-\eps$
for some $0\leq \eps<1$,
then
\[
 d_{\rm BM}(L,\Delta)\leq \frac{1}{1-\eps}.
\]
\end{theorem}
\unlinktheorem

Our bound is stronger, 
and is even asymptotically better.
Our bound holds for all $\eps \in [0,1)$ and has the form 
$\frac{1}{1-\eps}
= 1+\eps+O(\eps^2)$,
while the previous bound holds only for $\eps \leq \frac{1}{n}$ and is worse by a factor depending on the dimension
$1+\frac{(n+1)\eps}{1-n\eps} 
> 1 + (n+1) \eps$.

\subsubsection*{Outer additivity (\Cref{sec:outAdd})}
Several of the results above rely on the fact that simplices satisfy the following additivity property.
The Minkowski sum of two sets is
$A+B = \{a + b : a \in A,b \in B\}$.

\begin{defn}
A convex body $K \subset \R^n$ is called {\em outer additive} if for every $L = L_1+L_2$ where $L_1,L_2 \subset \R^n$ are convex and compact, it holds that
\[\lambda_{K}(L) = \lambda_{K}(L_1) + \lambda_{K}(L_2).\]
\end{defn}

The additivity of $K$ does not rely on the ``position'' of $K$; it is invariant under invertible affine maps.
Because it is always true that
$\lambda_{K}(L) \leq \lambda_{K}(L_1) + \lambda_{K}(L_2)$, the additivity of $K$ is equivalent to 
$\lambda_{K}(L) \geq \lambda_{K}(L_1) + \lambda_{K}(L_2)$. 

The fact that simplices are outer additive easily follows from the additivity of support functions.
It is natural to ask whether similar properties can be obtained by replacing
    the simplex with some other convex body.
\Cref{sec:outAdd} characterizes  simplices as the only outer additive convex bodies.

\linktheoremto{restatement:thm:AddIsSimplex}
\begin{theorem}
\label{thm:AddIsSimplex}
A convex body is outer additive iff it is a simplex.
\end{theorem}
\unlinktheorem

\subsubsection*{Depth complexity (\Cref{sec:depth})}

Our final application is motivated by the study of neural networks (NNs).
The study of NNs lead to the following model for constructing polytopes.
Polytopes of depth zero are 
\[
\cP_{n,0}=\{\{x\}:x\in\R^n\}.
\]
For $d>0$, polytopes of depth at most $d$ are
\[
\cP_{n,d}
=
\left\{
\sum_{j=1}^m \conv(K_j\cup L_j)
:
m\in\N,\ K_j,L_j\in\cP_{n,d-1}
\right\}.
\]
The \emph{depth complexity} of a polytope $P\subset\R^n$ is the smallest
$d$ such that $P\in\cP_{n,d}$.
We find this computational-complexity-like framework for studying polytopes interesting in its own right, and provide further motivations later on.

\linktheoremto{restatement:thm:Simplex}
\begin{theorem}
\label{thm:Simplex}
If a polytope $P \subset \R^n$ has depth complexity $d < \lceil \log_2(n+1) \rceil$, then 
\[\rho_\Delta(P) \le 2^d-1.\]
This bound is sharp (this inequality is an equality for some polytopes).
\end{theorem}
\unlinktheorem

Consequently, small depth polytopes are, in some sense, close to being symmetric in the Minkowski measure of symmetry. 

\linktheoremto{restatement:thm:minkowski-depth-upper}
\begin{theorem}
\label{thm:minkowski-depth-upper}
Let $P\subset\R^n$ be a non-singleton polytope, let $k=\dim P$, and suppose
that $P\in\cP_{n,d}$. If $0\le d<\lceil\log_2(k+1)\rceil$, then
\[
m^*(P)\le 2^d-1.
\]
\end{theorem}
\unlinktheorem

\Cref{thm:Simplex}
follows from the following two more informative theorems, which control the simplex-based measure under the two operations on polytopes we consider.
\linktheoremto{restatement:thm:De}
\begin{theorem}
\label{thm:De}
If $L = L_1+\ldots+L_m$ where $L_1,\ldots,L_m \subset \R^n$ are non-singleton compact convex sets, then
\[
\rho_\Delta(L) \leq  \max \{ \rho_\Delta(L_j) : j \in [m]\}.
\]
\end{theorem}
\unlinktheorem

\linktheoremto{restatement:thm:DeUnion}
\begin{theorem}
\label{thm:DeUnion}
If $L = \conv(L_1 \cup L_2)$ where $L_1,L_2 \subset \R^n$ are non-singleton compact convex sets, then
\[
\rho_\Delta(L) \le \rho_\Delta(L_1) + \rho_\Delta(L_2) + 1.
\]
\end{theorem}
\unlinktheorem

In words, the theorems state that Minkowski sum does not increase the simplex-based measure beyond the worst summand, and the convex hull of union of two polytopes at most doubles (with a possible extra $+1$). 
Both theorems are proved in \Cref{sec:rho-operations}.

Because $\rho_\Delta(\Delta) = n$,
we can deduce that simplices cannot be approximated by low-depth polytopes. 
In \Cref{sec:inapproximability}, we translate the geometric obstruction given by $\rho_\Delta(\Delta)$ being large to a natural inapproximability statement for $h_\Delta$ by small depth {\em input  convex neural networks}
(we mention that $h_\Delta$ is particularly important in this area---it is ``complete for depth'').

\subsection*{Acknowledgments}

We thank Gennadiy Averkov, Juan Luis Valerdi Cabrera and
Florestan Brunck
for helpful remarks. 
The authors are supported by a DNRF Chair grant.
The authors are part of BARC, Basic Algorithms Research Copenhagen, supported by the VILLUM Foundation grant 54451.

\section{Basic properties}
\label{sec:simplex-measure}

\subsection{A formula via support functions}
The goal of this section is to develop a formula for the simplex-based measure (\Cref{prop:rho-formula} below).
For the simplex $\Delta = \Delta^n \subset \R^n$ and a non-singleton compact convex set $L \subset \R^n$, recall from \Cref{subsection:simplex-based-measure} that
\[
\rho_\Delta(L):=\frac{\lambda_{-\Delta}(L)}{\lambda_\Delta(L)}.
\]
Let $\mathscr S$ be the family of nondegenerate (i.e., with nonempty interior) simplices in $\R^n$ whose barycenter is the origin.
A tuple $v=(v_0,\dots,v_n)\in (\R^n)^{n+1}$ is called \emph{nondegenerate} if any $n$ of the vectors $v_0,\dots,v_n$ are linearly independent.
Let $V$ be the set of nondegenerate 
$(v_0,\dots,v_n)\in(\R^n)^{n+1}$ such that
\[
\sum_{i=0}^n v_i=0.
\]
For every nondegenerate tuple $v=(v_0,\dots,v_n)\in V$, define
\[
\Delta_v:=\{x\in\R^n:\ \langle v_i,x\rangle\le 1 \ \forall i \in \{0,\dots,n\} \}.
\]
The vectors $v_i$ are normal to the facets of $\Delta_v$.

\begin{lemma}\label{lem:simplex}
If $v\in V$ then $\Delta_v\in\mathscr S$.
Conversely, for every $\Delta\in\mathscr S$, there exists a unique
$v\in V$ such that $\Delta=\Delta_v$.
\end{lemma}

\begin{proof}
Assume $v \in V$.
For each $i$, let $q_i \in \R^n$ be the unique point satisfying
\[
\langle v_j,q_i\rangle=1 \qquad \forall j\ne i.
\]
Hence,
$\langle v_i,q_i\rangle
=
-\sum_{j\ne i}\langle v_j,q_i\rangle
=-n$ for every $i$.
The points $q_0,\ldots,q_n$ are the vertices of the nondegenerate $\Delta_v$. Also, for every $j$,
\[
\Bigl\langle v_j,\sum_{i=0}^n q_i\Bigr\rangle
=\langle v_j,q_j\rangle+\sum_{i\ne j}\langle v_j,q_i\rangle
=-n+n=0.
\]
Since the $v_j$'s span $\R^n$, the barycenter of $\Delta_v$ is $\sum_i q_i = 0$. Thus $\Delta_v\in\mathscr S$.

Conversely, let $\Delta=\conv(q_0,\dots,q_n)\in\mathscr S$. Let $v_0,\ldots,v_n$ be the facet normals normalized such that
\[
\Delta=\{x\in\R^n:\ \langle v_i,x\rangle\le 1 \ \forall i\}.
\]
Thus, $\langle v_i,q_j\rangle=1$ for $j\ne i$, and 
\[
\langle v_i,q_i\rangle=-\sum_{j\ne i}\langle v_i,q_j\rangle=-n.
\]
Therefore, for every $j$,
\[
\Bigl\langle\sum_{i=0}^n v_i,q_j\Bigr\rangle
=
\langle v_j,q_j\rangle+\sum_{i\ne j}\langle v_i,q_j\rangle
=-n+n=0.
\]
Since the $q_j$ span $\R^n$, this gives $\sum_{i=0}^n v_i=0$. The nondegeneracy of $(v_0,\dots,v_n)$ holds because an $n \times n$ matrix with $-n$ on the diagonal and $1$ elsewhere has full rank, and uniqueness follow from the normalization.
\end{proof}

\begin{lemma}\label{lem:o-independent}
For every $(v_0,\dots,v_n)\in V$, every compact convex set $K \subset \R^n$, and every $o\in\R^n$,
\[
\sum_{i=0}^n h_{K-o}(v_i) = \sum_{i=0}^n h_K(v_i).
\]
\end{lemma}

\begin{proof}
Use that $h_{K-o}(u)=h_K(u)-\langle o,u\rangle$ and $\sum_{i=0}^n v_i=0$.
\end{proof}

\begin{lemma}\label{lem:lambda-formula}
For $v=(v_0,\dots,v_n)\in V$ and a nonempty compact convex $K \subset \R^n$,
\[
\lambda_{\Delta_v}(K)=\frac1{n+1}\sum_{i=0}^n h_K(v_i).
\]
\end{lemma}

\begin{proof}
The inclusion $K\subset t+\lambda\Delta_v$ is equivalent to
\[
h_K(v_i)\le \langle v_i,t\rangle+\lambda \qquad \forall i,
\]
because
\[
t+\lambda\Delta_v 
=
\{ x\in\R^n: \langle v_i,x-t \rangle \le \lambda \ \forall i\}.
\]
Averaging over $i$,
\[
\lambda\ge \frac1{n+1}\sum_{i=0}^n h_K(v_i) =: \lambda_0.
\]
Define the linear map $A:\mathbb R^n\to\mathbb R^{n+1}$ by
\[
A(t)=(\langle v_0,t\rangle,\ldots,\langle v_n,t\rangle).
\]
Since the vectors $v_0,\dots,v_n$ span $\R^n$, the map $A$ is injective.
Moreover, for every $t\in\R^n$,
\[
\sum_{i=0}^n \langle v_i,t\rangle
=
\Bigl\langle \sum_{i=0}^n v_i,t\Bigr\rangle
=0.
\]
It follows that the image of $A$ is
\[
\operatorname{im}(A)=H:=\Bigl\{(c_0,\dots,c_n)\in\R^{n+1}:\ \sum_{i=0}^n c_i=0\Bigr\}.
\]
For
$b_i:=h_K(v_i)-\lambda_0$
we have $\sum_{i=0}^n b_i=0$. 
Hence, there exists $t\in\R^n$ such that
\[
A(t)=(b_0,\dots,b_n),
\]
or
\[
h_K(v_i)\le \langle v_i,t\rangle+\lambda_0 \qquad \forall i. 
\]
It follows that $K\subset t+\lambda_0\Delta_v$ and
$\lambda_{\Delta_v}(K)=\lambda_0$.
\end{proof}

\begin{prop}\label{prop:rho-formula}
Let $v=(v_0,\dots,v_n)\in V$, and let $K \subset \R^n$ be a non-singleton compact convex set and $o\in\R^n$. Then,
\[
\rho_{\Delta_v}(K)=\frac{\sum_{i=0}^n h_{K-o}(-v_i)}{\sum_{i=0}^n h_{K-o}(v_i)}.
\]
\end{prop}

\begin{proof}
Since $-\Delta_v=\Delta_{-v}$ and $-v\in V$, Lemmas \ref{lem:o-independent} and \ref{lem:lambda-formula} complete the
proof.
\end{proof}

\subsection{Minkowski sum and convex hull of union}
\label{sec:rho-operations}

\begin{restatetheorem}{thm:De}[restatement]
If $L = L_1+\ldots+L_m$ where $L_1,\ldots,L_m \subset \R^n$ are non-singleton compact convex sets, then
\[
\rho_\Delta(L) \le \max \{ \rho_\Delta(L_j) : j \in [m]\}.
\]
\end{restatetheorem}

\begin{proof}
Write $\Delta=\Delta_v$ for $v=(v_0,\dots,v_n)\in V$. For $j\in[m]$, set
\[
A_j:=\sum_{i=0}^n h_{L_j}(-v_i)
\quad \text{and} \quad
B_j:=\sum_{i=0}^n h_{L_j}(v_i).
\]
By \Cref{prop:rho-formula},
\[
\rho_\Delta(L_j)=\frac{A_j}{B_j}
\qquad \forall j\in[m].
\]
By the linearity of the support function,
\[
\rho_\Delta(L)=\frac{\sum_{j=1}^m A_j}{\sum_{j=1}^m B_j}.
\]
By \Cref{lem:lambda-formula}, each $B_j=(n+1)\lambda_\Delta(L_j)$ is positive because $L_j$ is not a point. 
Hence $\rho_\Delta(L)$ is a weighted average of $\rho_\Delta(L_1),\dots,\rho_\Delta(L_m)$:
\[
\rho_\Delta(L)\le 
\frac{\sum_{j=1}^m B_j \rho_\Delta(L_j)}{\sum_{j=1}^m B_j} \leq \max_{j\in[m]}\rho_\Delta(L_j).
\qedhere
\]
\end{proof}

\begin{restatetheorem}{thm:DeUnion}[restatement]
If $L = \conv(L_1 \cup L_2)$ where $L_1,L_2 \subset \R^n$ are non-singleton compact convex sets, then
\[
\rho_\Delta(L) \le \rho_\Delta(L_1) + \rho_\Delta(L_2) + 1.
\]
\end{restatetheorem}

\begin{proof}
\label{proof:thm:DeUnion}
Write $\Delta=\Delta_v$ for $v=(v_0,\dots,v_n)$. 
Set
\[
A:=\sum_{i=0}^n h_L(-v_i),
\quad \text{and} \quad
B:=\sum_{i=0}^n h_L(v_i).
\]
For $j=1,2$, set
\[
A_j:=\sum_{i=0}^n h_{L_j}(-v_i),
\quad \text{and} \quad 
B_j:=\sum_{i=0}^n h_{L_j}(v_i),
\]
Because $L=\conv(L_1 \cup L_2)$, 
\[
h_L(u)=\max\{h_{L_1}(u),h_{L_2}(u)\}\qquad \forall u\in\R^n.
\]
Hence,
$B\ge B_1$ and
$B\ge B_2$.
Also, for every compact set $K$ and every $u\in\R^n$,
\[
h_K(u)+h_K(-u)\ge 0.
\]
For fixed $i$ such that $h_{L}(-v_i) = h_{L_1}(-v_i)$ (the other case $h_{L}(-v_i) = h_{L_2}(-v_i)$ is similar),
\begin{align*}
h_{L}(-v_i) &=  h_{L_1}(-v_i)
\\  &\le h_{L_1}(-v_i)+h_{L_2}(-v_i)+h_{L_2}(v_i)
\\  &\le  h_{L_1}(-v_i)+h_{L_2}(-v_i)+h_{L}(v_i).
\end{align*}
Summing over $i$ gives
\[
A\le A_1+A_2+B.
\]
By \Cref{prop:rho-formula},
\begin{align*}
\rho_\Delta(L)
&=\frac{A}{B}
\le \frac{A_1+A_2+B}{B}
\le \frac{A_1}{B_1}+\frac{A_2}{B_2}+1\\
&=\rho_\Delta(L_1)+\rho_\Delta(L_2)+1.
\qedhere
\end{align*}

\end{proof}

\section{Equivalence to the Minkowski measure}

\label{sec:minkowski-equivalence}

This section shows that an affine invariant version of the simplex-based measure is equivalent to the classical Minkowski measure.
Let $K \subset \R^n$ be a non-singleton compact convex set. 
Denote by $\operatorname{aff}(K)$ the affine span of $K$ and by $\operatorname{relint}(K)$ the relative interior of $K$.
Write $\operatorname{aff}(K)$ as
$\operatorname{aff}(K)=o+E_K$ where $E_K$ is a linear subspace
and $o \in K$.
Denote by $k$ the dimension of $E_K$.
We can think of $E_K$ as a $k$-dimensional Euclidean space and of $K-o$ as a subset of this space.
For an Euclidean space $E$, write
\[
\S(E):=\{u\in E:\|u\|_2=1\}.
\]
All statements from \Cref{sec:simplex-measure} apply verbatim in any Euclidean space; below we use them inside $E_K$.

We use the following characterization of Minkowski measure (\Cref{thm:minkowski-slabs}) inside $E_K$.
For a non-singleton compact convex set $K \subset \R^n$ and $o\in\operatorname{relint}(K)$, the Minkowski symmetry of $K$ at $o$ is
\[
m_K(o)=\max_{u\in\S(E_K)}\frac{h_{K-o}(-u)}{h_{K-o}(u)}.
\]
The Minkowski measure of $K$ is
\[
m^*(K)=\inf_{o\in\operatorname{relint}(K)} m_K(o).
\]

For full-dimensional $K$, this is the classical Minkowski measure of symmetry.
For lower-dimensional $K$, it is the same quantity in the affine hull of $K$.

The following is an affine invariant version of the simplex-based measure.
\begin{defn}
For a non-singleton compact convex set $K \subset \R^n$, let $\mathscr S(K)$ denote the family of nondegenerate simplices in $E_K$ whose barycenter is $0$, and define
\[
\alpha(K):=\sup_{\Delta\in\mathscr S(K)}\rho_\Delta(K).
\]
\end{defn}

\begin{corollary}\label{cor:alpha-formula}
Let $K \subset \R^n$ be a non-singleton compact convex set, let $E=E_K$, let $k=\dim E_K$, and let $o\in \operatorname{aff}(K)$.
Define
\[
V(E):=\Bigl\{(v_0,\dots,v_k)\in E^{k+1}:\ \sum_{i=0}^k v_i=0 \Bigr\}.
\]
Then,
\[
\alpha(K)=\max_{v\in V(E)}
\frac{\sum_{i=0}^k h_{K-o}(-v_i)}{\sum_{i=0}^k h_{K-o}(v_i)}.
\]
\end{corollary}

\begin{proof}
By the results of Section~\ref{sec:simplex-measure},
\[
\alpha(K)=
\sup_{\substack{v\in V(E)\\ v\ \text{nondegenerate}}}
\frac{\sum_{i=0}^k h_{K-o}(-v_i)}{\sum_{i=0}^k h_{K-o}(v_i)}.
\]
Choose $z\in\operatorname{relint}(K)$. By Lemma~\ref{lem:o-independent} applied in $E$,
\[
\sum_{i=0}^k h_{K-o}(v_i)=\sum_{i=0}^k h_{K-z}(v_i).
\]
Since $0\in \operatorname{int}_{E}(K-z)$, we have $h_{K-z}(u)>0$ for every $u\in \S(E)$. Hence, the quotient defines a continuous function on a compact set (because $\rho_{\Delta_{\lambda v}}(K)
= \rho_{\Delta_{v}}(K)$ for $\lambda >0$). Since the 
nondegenerate tuples are dense in $V(E)$, the supremum above is the maximum over all $v\in V(E)$.
\end{proof}

\begin{lemma}\label{lem:caratheodory}
Let $K\subset \R^n$ be a non-singleton compact convex set, let $E=E_K$, and assume $0\in\operatorname{relint}(K)$ is a Minkowski center. Put $m:=m_K(0)=m^*(K)$ and define
\[
A:=\{u\in\S(E):\ h_{K}(-u)=m\,h_{K}(u)\}.
\]
Then, $0\in\conv(A)$.
\end{lemma}

\begin{proof}
 Assume towards a contradiction that $0\notin\conv(A)$.  
 By strict separation in the Euclidean space $E$, there are $p\in E$ and $\delta>0$ such that
\[
\langle p,u\rangle\ge\delta \qquad \forall u\in \conv(A).
\]
For small $t>0$, let $o:=-tp\in \operatorname{int}_{E}(K)=\operatorname{relint}(K)$. We will show that $o$ is a ``better candidate'' for a Minkowski center than $0$
(hence the contradiction). For every $u \in \S(E),$
\[
h_{K-o}(u)=h_K(u)+t\langle p,u\rangle
\quad \text{and} \quad
h_{K-o}(-u)=h_K(-u)-t\langle p,u\rangle,
\]
so
\begin{align*}
h_{K-o}(-u)-m h_{K-o}(u)
&= \bigl(h_K(-u)-mh(u)\bigr)-(1+m)t\langle p,u\rangle
\\ &:=
a(u) - tb(u),
\end{align*}
where
\[
a(u):= h_K(-u)-mh(u),
\quad \text{and} \quad
b(u):=(1+m)\langle p,u\rangle.
\]
Consider the sign of the expression $a(u) - tb(u)$.
For $u\in A$, it is negative, because $a(u)=0$ and $b(u)>0$.
For $u\in \S(E)\setminus A$, we have $a(u)<0$. By
continuity, $b(u)>0$ on some open neighborhood $U$ of $A$. On
the compact set $\S(E)\setminus U$, the function $a(u)$ is uniformly bounded
away from $0$.
It follows that, for a sufficiently small $t>0$, we have
$m_K(o) < m$.
\end{proof}

\begin{restatetheorem}{thm:main-equality}[restatement]
For every non-singleton compact convex set $K\subset \R^n$,
\[
\alpha(K)=m^*(K).
\]
\end{restatetheorem}

\begin{proof}
Let $E=E_K$ and $k=\dim K$. Choose a Minkowski center $o\in\operatorname{relint}(K)$ and translate so that $o=0$ so that $K\subset E$ and
\[
m:=m^*(K)=\max_{u\in\S(E)}\frac{h(-u)}{h(u)},
\]
where $h:=h_K$.
By Corollary~\ref{cor:alpha-formula},
\[
\alpha(K)=\max_{v=(v_0,\dots,v_k)\in V(E)}
\frac{\sum_{i=0}^k h(-v_i)}{\sum_{i=0}^k h(v_i)}.
\]
Because $h(-u)\le m h(u)$ for every $u\in E$, 
\[\alpha(K)\le m.\]
For the reverse inequality, let
\[
A:=\{u\in\S(E):\ h(-u)=m h(u)\}.
\]
By Lemma~\ref{lem:caratheodory}, $0\in\conv(A)$. By Carathéodory's theorem in the $k$-dimensional Euclidean space $E$, there exist $u_0,\dots,u_k\in A$ and $\lambda_0,\dots,\lambda_k\ge0$ such that $\sum_{i=0}^k \lambda_i=1$ and
\[
\sum_{i=0}^k \lambda_i u_i=0.
\]
Set
$w_i:=\lambda_i u_i$ so that
$w=(w_0,\dots,w_k)\in V(E)$. Since $u_i\in A$,
\[
h(-w_i)=\lambda_i h(-u_i)=m\lambda_i h(u_i)=m h(w_i)
\qquad \forall i.
\]
Summing gives
\[
\sum_{i=0}^k h(-w_i)=m\sum_{i=0}^k h(w_i).
\]
Corollary~\ref{cor:alpha-formula} implies that $\alpha(K)\ge m$.
\end{proof}

\section{Stability of Minkowski measure}
\label{sec:stability}

Recall that $\lambda_K(L)$ is the size of the smallest homothet of $K$ containing $L$.
Define the inner coefficient of $L$ with respect to $K$ by
\[
\mu_K(L)=\sup\{\lambda_K(H): H\in \cH_K,\ H\subseteq L\}.
\]
In words, it is the size of the largest homothet of $K$ contained in $L$. 
For a simplex $\Delta$ and a convex body $L$, set
\[
d_\Delta(L):=\frac{\lambda_\Delta(L)}{\mu_\Delta(L)}.
\]
For a non-singleton compact convex set $L$, define also
\[
d_{\rm BM}(L,\Delta):=\inf_{\Delta\in\mathscr S(L)} d_\Delta(L).
\]
Equivalently, $d_{\rm BM}(L,\Delta)$ is the Banach--Mazur distance from $L$ to the class of simplices in its affine hull.
In the following two subsections we relate $d_\Delta$ to the simplex-based measure $\rho_\Delta$.

The main result of this section is the following comparison between
$\rho_\Delta$ and $d_\Delta$.

\begin{restatetheorem}{thm:banach}[restatement]
Let $\Delta \subset \R^n$ be a simplex and let $L \subset \R^n$ be a convex
body. Then
\[
\frac{n}{d_\Delta(L)}
\le
\rho_\Delta(L)
\le
n-1+\frac{1}{d_\Delta(L)}.
\]
\end{restatetheorem}

The lower and upper bounds (\Cref{lem:IOtoE} and \Cref{lem:IOe-third}) are proved separately in the next two subsections.
For the stability estimate we only use the upper bound.

\begin{remark*}
    
Since all simplices of a fixed dimension are affinely equivalent, it suffices in the proofs below to work with the model simplex $\Delta^n$.

\end{remark*}

\subsection{A lower bound}

\begin{lemma}
\label{lem:IOtoE}
Let $\Delta=\Delta^n$ and let $L \subset \R^n$ be a convex body.
Then,
\[
\rho_\Delta(L) \ge \frac{n}{d_\Delta(L)}.
\]
\end{lemma}

\begin{proof}
Write $\Delta$ as $\Delta = \Delta_v$ for $v=(v_0,\dots,v_n)\in V$ using \Cref{lem:simplex}.
Let 
$K=t+\mu \Delta \in \cH_\Delta$ be such that $K \subseteq L$ where $\mu:=\mu_{\Delta}(L)$.
In words, $K$ is a largest homothet of $\Delta$ inside $L$.
Set
\[
A:=\sum_{i=0}^n h_L(-v_i)
\quad \text{and} \quad
B:=\sum_{i=0}^n h_L(v_i).
\]
By \Cref{prop:rho-formula} and \Cref{lem:lambda-formula},
\[
\rho_\Delta(L)=\frac{A}{B}
\quad\text{and}\quad
B=(n+1)\lambda_\Delta(L).
\]
Since $K\subseteq L$, we have $h_L(-v_i)\ge h_K(-v_i)$ for every $i$, which implies
\[
A\ge \sum_{i=0}^n h_K(-v_i).
\]
Because $\sum_i v_i=0$,
\[
\sum_{i=0}^n h_K(-v_i)
=
\sum_{i=0}^n \bigl(\langle -v_i,t\rangle+\mu h_{\Delta}(-v_i)\bigr)
=\mu\sum_{i=0}^n h_{\Delta}(-v_i),
\]
By \Cref{lem:lambda-formula},
\[
n = \lambda_{-\Delta}(\Delta)=\frac1{n+1}\sum_{i=0}^n h_\Delta(-v_i).
\]
Hence,
$A\ge n(n+1)\mu$. Finally,
\[
\rho_\Delta(L)=\frac{A}{B}
\ge \frac{n(n+1)\mu}{(n+1)\lambda_\Delta(L)}
= \frac{n}{d_\Delta(L)}.
\qedhere
\]
\end{proof}

\subsection{An upper bound}

\begin{lemma}
\label{lem:IOe-third}
For every convex body $L \subset \R^n$,
\[
\rho_\Delta(L) \le n-1 + \frac{1}{d_\Delta(L)}.
\]
\end{lemma}

Before proving the lemma, we need some preparation.
A supporting half-space $T$ of a convex body $L \subseteq \R^n$ is of the form
\[ T= \{x \in \R^n : \ip{u}{x} \leq h_L(u) \}\]
for some $u \in \S^{n-1}$.
It always holds that $L \subseteq T$.

\begin{lemma}
\label{lem:inscribed}
Let $L\subset\mathbb{R}^n$ be a convex body.
Then, there exist $k \leq n+1$ supporting half-spaces $T_1,\ldots,T_k$ of $L$ such that
\[
L \subseteq A := \bigcap_{j \in [k]} T_j
\quad 
\text{and}
\quad 
\mu_{\Delta}(L) = \mu_{\Delta}(A).
\]
\end{lemma}

The proof of the lemma uses techniques from the work of Klain \cite{klain2010containment}.
For a convex body $L \subset \R^n$,
we denote by $\Delta_i(L)$
a simplex $H \in \cH_{\Delta^n}$ such that
 $H \subseteq L$ and $\lambda_{\Delta^n}(H) = \mu_{\Delta^n}(L)$
 (it is not necessarily unique).
 In words, $\Delta_i(L)$ is a largest copy of $\Delta$ inside $L$.

\begin{proof}[Proof of \Cref{lem:inscribed}]
Let $\Delta_i = \Delta_i(L)$.
Because $\Delta_i \subseteq L$,
\[
h_L(u) \geq h_{\Delta_i}(u) 
\quad\forall u\in \S^{n-1}.
\]
Let
\[
U = \bigl\{ u\in \S^{n-1} : h_L(u) = h_{\Delta_i}(u)  \bigr\}.
\]

We claim that $0\in \conv(U)$. Otherwise, by convex separation, there exist $\eps>0$ and $v\in \S^{n-1}$ such that
$\langle v,u\rangle < - \eps$ for all $u\in U$.  
In other words, $U$ is contained in the (open) half-space $W:=\{w \in \R^n : \ip{v}{w} < -\eps\}$.
The set $\S^{n-1} \setminus W$ is non-empty, compact and is disjoint from $U$ so there is $\eps' > 0$ such that
\begin{align*}
h_L(u) \geq h_{\Delta_i}(u) + \eps'
\quad\forall u\in \S^{n-1} \setminus W.
\end{align*}
Hence, there is $\delta>0$ such that
\begin{align*}
h_L(u) > h_{\Delta_i}(u) + \delta\langle v,u\rangle = h_{\Delta_i + \delta v}(u)
\quad\forall u\in \S^{n-1}.
\end{align*}
It follows that $\Delta_i + \delta v$ is contained in the interior of $L$, and therefore $L$ contains a larger homothet of~$\Delta$, which is a contradiction. 

Now, by Carath\'eodory's  theorem, there exist $u_1,\dots,u_k \in U$
and $\alpha_1,\ldots,\alpha_k > 0$ for
$2 \leq k\leq n+1$ such that 
\[
\sum_{j} \alpha_j u_j = 0.
\]
Define the supporting half-spaces as (for $j \in [k]$)
\[
T_j
=\{ x : \langle x,u_j\rangle \leq h_L(u_j)\}
\]
and their intersection
\[
A=\bigcap_j T_j \supseteq L.
\]
It remains to prove that
$\mu_{\Delta}(L) = \mu_{\Delta}(A)$.
Because $L \subseteq A$,
we have $\mu_{\Delta}(L) \leq \mu_{\Delta}(A)$.
Assume towards a contradiction
that
\[(1+\gamma) \mu_{\Delta}(L) \leq  \mu_{\Delta}(A)\]
for some $\gamma > 0$ so that 
\[(1+\gamma) \Delta_i + t \subseteq A\]
for some $t \in \R^n$.
For every $j \in [k]$, 
\[
(1+\gamma) h_{\Delta_i}(u_j) + \ip{t}{u_j}
= h_{(1+\gamma) \Delta_i+t}(u_j)
\leq h_A(u_j)
= h_L(u_j) = h_{\Delta_i}(u_j)
\]
so 
\[
\gamma h_{\Delta_i}(u_j) + \ip{t}{u_j} \leq 0 .
\]
Since $\Delta_i$ has non-empty interior, choose
$z\in\operatorname{int}\Delta_i$. Because
$h_{\Delta_i}(u_j)>\langle z,u_j\rangle$ for every~$j$, we deduce the contradiction:
\[
0 = -\frac{1}{\gamma} \sum_j \alpha_j \ip{t}{u_j}
\geq 
\sum_j \alpha_j h_{\Delta_i}(u_j)
>
\sum_j \alpha_j \langle z,u_j\rangle
=0.
\qedhere
\]
\end{proof}

We also use the following Kadets-type result by Akopyan and Karasev~\cite{akopyan2012kadets}.
It is based on the notion of inductive covering (which is defined as follows).
The cover of $\R^n$ by a single set $V_1 = \R^n$ is inductive.
A cover $\R^d = V_1\cup\cdots\cup V_k$ is inductive if $V_1,\ldots,V_k$ are closed convex sets and for each $j \in [k]$, there exists an inductive covering
	\[
	\R^d = W_1\cup\cdots\cup W_{j-1}\cup W_{j+1}\cup\cdots\cup W_k 
	\]
	such that
\[W_\ell \subseteq  V_\ell \cup V_j
\qquad  \forall \ell \neq j.\]
The following\footnote{In \cite{akopyan2012kadets} the theorem is stated for partitions instead of coverings. The remark before definition~5 in that paper states that the theorem holds for coverings as well. In addition, their notation is different from ours: $\mu_{B}(C)$ is denoted by $r_B(C)$.}
 is Theorem 1 in \cite{akopyan2012kadets}.

\begin{theorem}
\label{thm:inductive-kadets}
Let $V_1,\ldots,V_k$ be an inductive covering of $\R^n$.
	Let $B\subset\mathbb{R}^n$ be a convex body, and let $C_j = V_j\cap B$ for each~$j \in [k]$. Then,
	\[
	\sum_j \mu_{B}(C_j) \ge 1.
	\]
\end{theorem}

\begin{proof}[Proof of \Cref{lem:IOe-third}]
Let $\Delta = \Delta^n$. 
We assume without loss of generality that $L$ is 
smooth and strictly convex, 
because the collection of such convex bodies is dense. 
We also assume without loss of generality that $\lambda_{\Delta}(L) = 1$ and prove that $\rho_\Delta(L) \le n-1 + \mu_{\Delta}(L).$

Denote by $\Delta_o$ a smallest homothet of $\Delta$ containing $L$
and by $\Delta_i$ a largest homothet of $\Delta$ contained in $L$.
Consider the set
\[
U = \bigl\{ u\in \S^{n-1} : h_L(u) = h_{\Delta_i}(u)  \bigr\}
\] from the proof of \Cref{lem:inscribed}.
Because $L$ is strictly convex, for every $u \in U$, the face $\Delta_i \cap \mathcal{S}(\Delta_i,u)$ of $\Delta_i$ is a vertex which we denote by $g(u) \in V(\Delta_i)$ (otherwise, the boundary of $L$ contains a line segment).
The vector $u$ is a normal to $L$ at $g(u)$. Because $L$ is strictly convex and smooth, the map $u \mapsto g(u)$ is injective. 
Choose the directions $u_j$ and the half-spaces 
\[
T_j
=\{ x : \langle x,u_j\rangle \leq h_L(u_j)\}
\]
from the proof of \Cref{lem:inscribed}. 
Let 
\[A'_0 = \bigcap_{j} T_j\]
and 
\[A'_j = cl( \R^n \setminus T_j)\]
where $cl$ denotes closure. 
Let $A_j = A'_j \cap \Delta_o$ for $j \in \{0,1,\ldots,k\}$.
It follows that 
\begin{align*}
\R^n &= \bigcup_j A'_j
\intertext{and}
\Delta_o &= \bigcup_j A_j.
\end{align*}
By \Cref{lem:inscribed},
\[\mu_{\Delta}(L) = \mu_{\Delta}(A_0).\]

Fix $j \in [k]$ for now. The direction $u_j \in U$ defines the vertex $g_j := g(u_j)$ of $\Delta_i$.
Denote by $g'_j$ the vertex of $\Delta_o$
that corresponds to $g_j$: 
\[\{g'_j\} =  \mathcal{S}(\Delta_o,u_j) \cap \Delta_o.\]
For each $j$, let $x \in \mathcal{S}(\Delta_o,-u_j) \cap \Delta_o$.
Because $\lambda_{\Delta}(L) = 1$ and the interval $[g'_j,x]$ is contained in $\Delta_o$,
\[
\mu_{\Delta}(A_j) \leq \mu_{[g'_j, x]}(A_j) .
\]
The interval $[g'_j,x]$ intersects $\mathcal{S}(\Delta_i,u_j)$ at a point $y$. 
Using \eqref{eq:segment}, we can deduce
\[
\mu_{[g'_j, x]}(A_j) \leq \frac{|[g'_j, y]|}{|[g'_j, x]|}.
\]
The homothety with center $g'_j$ that moves $x$ to $y$ also moves $\Delta_o$, which has
coefficient of homothety $\lambda_{\Delta}(\Delta_o)=1$, to a homothet $H_j$ with coefficient of homothety $\lambda_{\Delta}(H_j) = \frac{|[g'_j, y]|}{|[g'_j, x]|}$. 
Thus
\[\mu_{\Delta}(A_j) \leq \lambda_\Delta(H_j).\]

Now, collect the information on all $j$'s. 
Because $u \mapsto g(u)$ is injective,
$g'_1,\ldots,g'_k$ are distinct. 
Write $\Delta_o=t+\Delta_v$, where $v=(v_0,\dots,v_n)\in V$ is a nondegenerate tuple and the vertex $g'_j$ corresponds to the index $i(j)$. Define
\[
\beta_i:=\frac{h_{\Delta_o}(-v_i)-h_L(-v_i)}{n+1}
\quad \forall i \in \{0,\dots,n\}.
\]
By \Cref{prop:rho-formula} and the normalization $\lambda_\Delta(L)=\lambda_\Delta(\Delta_o)=1$,
\[
\sum_{i=0}^n \beta_i = n-\rho_\Delta(L).
\]
Moreover, $H_j$ is contained in the corner of $\Delta_o$ based at $g'_j$, so $\lambda_\Delta(H_j)\le \beta_{i(j)}$. Therefore
\[
\sum_j \mu_{\Delta}(A_j) \le \sum_j \beta_{i(j)} \le n-\rho_\Delta(L).
\]

Next, we show that we can use the Kadets-type result by Akopyan and Karasev.

\begin{claim}
\label{clm:A'isInd}
The covering $A'_0\cup A'_1\cup \dots \cup A'_k$ of $\R^n$ 
is inductive.
\end{claim}
\begin{proof}
We prove a more general claim.
For $J \subseteq [k]$, define the $J$-covering to be $V_0$ and $V_j$ for $j \in J$, where
\[V_0 = \bigcap_{j \in J} T_j\]
(if $J = \emptyset$ then $V_0 = \R^n$)
and for $j \in J$,
\[V_j = cl( \R^n \setminus T_j).\]
We prove by induction that the $J$-covering is an inductive covering. 
The $\emptyset$-covering is an inductive covering
(it has one set).
It remains to perform the inductive step.
Let $V_0$ and $V_j$ for $j \in J$ be the $J$-covering for $|J|>0$.
We need to verify that the relevant condition holds for every $j \in \{0\} \cup J$.

First, consider the case $j \neq 0$.
Consider the $(J\setminus \{j\})$-covering: 
\[W_0 = \bigcap_{\ell \in J\setminus \{j\}} T_\ell \]
and $W_\ell = V_\ell$ for $\ell \in J\setminus \{j\}$.
For $\ell = 0$, we use that
$V_0 = cl(W_0 \setminus V_{j})$ and get
\[
W_0 \subseteq cl(W_0 \setminus V_{j}) \cup V_{j} = V_0 \cup V_{j}.
\]
For $\ell \in J \setminus \{j\}$,
\[
W_\ell = V_\ell \subseteq V_\ell \cup V_{j}.
\]

Second, consider the case $j = 0$.
Choose some $i \in J$ and
consider the $(J \setminus \{i\})$-covering:
\[W_{i} = \bigcap_{\ell \in J\setminus \{i\}} T_\ell \]
and $W_\ell = V_\ell$ for 
$\ell \in J\setminus \{i\}$.
For $\ell = i$, we use that
$V_0 = cl(W_{i} \setminus V_{i})$ and get
\[W_{i} \subseteq cl(W_{i} \setminus V_{i}) \cup V_{i} = V_0 \cup V_{i}.
\]
For $\ell \in J \setminus \{i\}$,
\[
W_\ell = V_\ell \subseteq V_\ell \cup V_{0}.
\qedhere 
\]
\end{proof}

Putting it all together, by \Cref{clm:A'isInd} and \Cref{thm:inductive-kadets},
\[ 
\mu_{\Delta}(L) + n-\rho_\Delta(L) \geq \sum_{j \in \{0,1,\ldots,k\}} \mu_{\Delta}(A_j) \geq 1,
\]
which is equivalent to $\rho_\Delta(L) \le n-1 + \mu_{\Delta}(L)$.
This proves the lemma under the normalization $\lambda_\Delta(L)=1$, and the general statement follows by scaling.
\qedhere

\end{proof}

\begin{remark*}
Equality is achieved in \Cref{lem:IOe-third}
in the following example of $L$. Denote by $v_1,\ldots,v_{n+1}$ the vertices of $\Delta^{n}$.
Let $a$ and $b$ be two points on the edge $[v_1,v_2]$, 
and let $L$ be $\conv\{a,b,v_3,\dots,v_{n+1}\}$.
\end{remark*}

\subsection{Stability.}

Now we bound the stability of the Minkowski measure. 

\begin{lemma}
\label{lemma:minkowski-bm-simplex}
For every convex body $L\subset \R^n$,
\[
\frac{n}{d_{\rm BM}(L,\Delta)} \le m^*(L) \le n-1+\frac{1}{d_{\rm BM}(L,\Delta)}.
\]
\end{lemma}

\begin{proof}
By \Cref{thm:main-equality},
\[
m^*(L)=\alpha(L)=\sup_{\Delta\in\mathscr S(L)} \rho_\Delta(L).
\]
By \Cref{lem:IOtoE} and \Cref{lem:IOe-third}, for every $L$,
\[
\frac{n}{d_\Delta(L)} \le \rho_\Delta(L) \le n-1+\frac{1}{d_\Delta(L)}.
\]
The last observation is
\[
\sup_{\Delta\in\mathscr S(L)} \frac{1}{d_\Delta(L)}
=
\frac{1}{\inf_{\Delta\in\mathscr S(L)} d_\Delta(L)}
=
\frac{1}{d_{\rm BM}(L,\Delta)} . \qedhere
\]
\end{proof}

\begin{restatetheorem}{thm:stabBM}[restatement]
If $L\subset \R^n$ is a convex body and
$m^*(L)\ge n-\eps$
for some $0\leq \eps<1$, then
\[
 d_{\rm BM}(L,\Delta)\leq \frac{1}{1-\eps}.
\]
\end{restatetheorem}

\begin{proof}
   It immediately follows from the lemma; indeed, 
by assumption and the lemma,
\[
    n-\eps \le m^*(L) \le n-1+\frac{1}{d_{\rm BM}(L,\Delta)}.
    \qedhere
\]
\end{proof}

\section{Outer additivity}
\label{sec:outAdd}

Fix $n$ and let $\Delta = \Delta^{n}$.
Denote the vertices of $\Delta$ by $v_1,\ldots,v_{n+1}$.
For each $j \in [n+1]$, let $F_j$ be the hyperplane that contains the facet that is opposite to~$v_j$. Call $F_1,\ldots,F_{n+1}$ the facet hyperplanes of $\Delta$. Let $u_j$ be the unit vector and $b_j \in \R$ be such that
$F_j  = \{x\in\mathbb{R}^n:\ip{u_j}{x} = b_j\}$
and $\ip{u_j}{v_j} > b_j$.

For $u \in \S^{n-1}$ and a convex body $K \subset \R^n$, denote by $\mathcal{S}(K,u)$
the supporting hyperplane of $K$ in direction $u$, which is of the form 
\[\mathcal{S}(K,u)= u^\perp + h_K(u) u.\]

\subsection{Simplices are outer additive}

Recall that a convex body $K\subset \R^n$ is outer additive if
\[
\lambda_K(L_1+L_2)=\lambda_K(L_1)+\lambda_K(L_2)
\]
for all convex compact $L_1,L_2\subset \R^n$.
Also recall the following well-known additivity of support functions: 
$h_{L_1+L_2} = h_{L_1} + h_{L_2}$.

\begin{claim}
\label{clm:SimAdd}
All simplices are outer additive.
\end{claim}

\begin{proof}
Let $L = L_1+L_2$ for convex compact $L_1$ and $L_2$.
The simplex $\Delta = \Delta^n$ has $n+1$ facets with unit normals $u_1,\ldots,u_{n+1}$. 
The $n+1$ supporting hyperplanes $\mathcal{S}(L,-u_1), \ldots,\mathcal{S}(L,-u_{n+1})$ define the smallest homothet of $\Delta$ that contains $L$; denote it by $\Delta_o(L)$.
The same construction defines $\Delta_o(L_1)$ and $\Delta_o(L_2)$.
Because support functions are additive,
\begin{equation}
\label{eqn:DeltaO}
\Delta_{o}(L)
= \Delta_{o}(L_1)+\Delta_{o}(L_2).
\qedhere    
\end{equation}
\end{proof}

\subsection{Outer additive bodies are simplices}

For $u \in \S^{n-1}$ and a convex body $K \subset \R^n$,
let
\begin{align*}
A_u = A_u(K)&= \mathcal{S}(K,-u) \cap K
\intertext{and}
B_u = B_u(K)&= \mathcal{S}(K,u) \cap K.
\end{align*}
For $x,y \in \R^n$, denote by $[x,y]$
the interval between $x,y$.
The main property we use is that
\begin{align}
\label{eq:segment}
a \in A_u, \ b \in B_u \quad \Longrightarrow \quad
\lambda_{K}([a,b]) = 1.
\end{align}
For $x \in \partial K$,
denote by $\Phi(x) = \Phi_K(x)$ the collection of $y \in \partial K$ such that 
$[x,y] \subset \partial K$.

The key step is the following geometric lemma that gives a necessary ``local'' condition for being outer additive. 

\begin{lemma}
\label{lem:notAdditive}
Let $K \subset \R^n$ be a convex body.
If there are $u \in \S^{n-1}$, $a \in A_u$, $b \in B_u$ and $x \in \partial K \setminus (\Phi(a) \cup \Phi(b))$ then $K$ is not outer additive.
\end{lemma}

\begin{proof}
Suppose the condition is satisfied with $u,a,b,x$; see illustration in \Cref{fig:notAdditive}.
Let $u' \in \S^{n-1}$ be such that $x\in A_{u'}$
and let $y \in B_{u'}$. 
Because $a, x, y \in K$, 
\[[ \tfrac{a+y}{2}, \tfrac{a+x}{2}] \subset K.\] 
By choice of $x$, it holds that
$[a,x] \not \subseteq \partial K$. So, $[a,x]$ must contain an interior point of $K$ which implies  \[\tfrac{a+x}2 \not \in \partial K.\] 
We can conclude that there is $\alpha_1 >1$ such that
\[[ \tfrac{a+y}2, \tfrac{a+y}2 + \tfrac{\alpha_1(x-y)}2 ] \subset K .\]
By the same argument for the points $b, x, y$, 
there is $\alpha_2 > 1$ so that
\[[ \tfrac{b+y}{2}, \tfrac{b+y}{2} + \tfrac{\alpha_2(x-y)}2 ] \subset K.\]
Setting $\alpha = \min\{\alpha_1, \alpha_2\}$, we conclude
that the following 
parallelogram $L$ lies in $K$:
\[
L = \conv \{ 
\tfrac{a+y}2, \tfrac{a+y}2 + \tfrac{\alpha(x-y)}2, 
\tfrac{b+y}2, \tfrac{b+y}2 + \tfrac{\alpha(x-y)}2
\}.
\]
It can be expressed as $L = L_1 + L_2$ with
\[
L_1 = [ \tfrac{a+y}2, \tfrac{b+y}2 ] 
\text{\ \ \ \ and\ \ \ \ } 
L_2 = [ 0, \tfrac{\alpha(x-y)}2 ] .
\]
Because $L \subset K$, 
\[\lambda_{K}(L) \leq 1.\]
The interval $L_1$ is parallel to $[a,b]$. 
Because $a \in A_u$ and $b \in B_u$,
\[\lambda_{K}([a,b]) = 1,\]
so
\[\lambda_{K}(L_1) = \lambda_{K}(\tfrac{1}{2}[a,b]) = \tfrac{1}{2}.\]
By the same argument
applied to $L_2$,
\[\lambda_{K}(L_2) = \lambda_{K}(\tfrac{\alpha}{2}[x,y]) = \tfrac{\alpha}{2}.\]
Because $\alpha >1$, it follows that $K$ is not outer additive.
\end{proof}

\begin{figure}\centering
\includegraphics{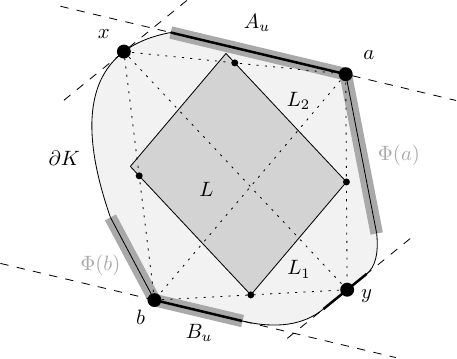}
\caption{The ``local'' condition for being outer additive from \Cref{lem:notAdditive}.
}
\label{fig:notAdditive}
\end{figure}

\begin{restatetheorem}{thm:AddIsSimplex}[restatement]
A convex body is outer additive if and only if it is a simplex.
\end{restatetheorem}

\begin{proof}
We already know that simplices are outer additive.
It remains to prove that outer additive convex bodies are simplices. 
The proof is by induction on~$n$.
The case $n=1$ is trivial.
For $n>1$, assume that $K$ is outer additive. 
Let $u \in \S^{n-1}$ be such that $B_u$ has maximum dimension.
Let $a \in A_u$.

\medskip 
\begin{claim}
\label{clm:partialK}
If $\partial K \subseteq B_u \cup \Phi(a)$
then $K = \conv(\{a\} \cup B_u)$.    
\end{claim}

\begin{proof}
Assume $\partial K \subseteq B_u \cup \Phi(a)$.
The inclusion $K \supseteq \conv(\{a\} \cup B_u)$ is trivial. It remains to prove the other inclusion.
Let $p$ be a point in the interior of $K$.
Let $F$ be a two-dimensional plane containing $a$ and $p$. Denote by $L$ the two-dimensional convex body $L = F \cap K$, so  $\partial L = F \cap \partial K$. By assumption, $\partial L \subseteq (F \cap B_u) \cup (F \cap \Phi(a))$.
For every $x \in F \cap \Phi(a)$, we have
$[a,x] \subset \partial L$. So $F \cap \Phi(a)$ comprises two line segments. Consequently, $F \cap B_u$ is a closed interval of positive length, and $L$ is the triangle $\conv(\{a\} \cup (F \cap B_u)) = F \cap (\conv(\{a\} \cup B_u))$.
So, $p \in \conv(\{a\} \cup B_u)$ and consequently $K \subseteq \conv(\{a\} \cup B_u)$.
\end{proof}

\begin{claim}
\label{clm:dimBu}
The dimension of $B_u$ is $n-1$.
\end{claim}

\begin{proof}
If $\partial K \subseteq B_u \cup \Phi(a)$
then by \Cref{clm:partialK} we know that $K = \conv(\{a\} \cup B_u)$ which implies the claim because $K$ has full dimension. 
Otherwise, there is $x \in \partial K$ such that
$x \not \in B_u \cup \Phi(a)$.
Because $K$ is outer additive, by \Cref{lem:notAdditive}, for all $b \in B_u$,
we know that $x \in \Phi(b)$.
Assume towards a contradiction that the dimension of $B_u$ is smaller than $n-1$.
Let $F$ be a hyperplane so that $x \in F$ and $B_u \subseteq F$.
It follows that $F$ is a supporting hyperplane of $K$ and that $F \cap K$ has dimension larger than that of $B_u$ which is a contradiction. 
\end{proof}

\begin{claim}
If $b$ is in the relative interior of $B_u$ then $\Phi(b) = B_u$.
\end{claim}

\begin{proof}
The inclusion $B_u \subseteq \Phi(b)$ is trivial.
It remains to prove the other inclusion. 
Assume towards a contradiction that $x \in \Phi(b) \setminus B_u$ so that $[b,x] \subset \partial K$. Because $b$ is in the relative interior of $B_u$ and by \Cref{clm:dimBu},
every continuous path on $\partial K$ from $b$
to outside $B_u$ has an open part inside $B_u$. In particular,
there is $y$ in the interior of $[b,x]$ such that $[b,y] \subset B_u \subset \mathcal{S}(K,u)$. This implies that 
$[b,x] \subset \mathcal{S}(K,u) \cap K = B_u$, which is a contradiction.
\end{proof}

Let $b$ be in the relative interior of $B_u$ so that $\Phi(b) = B_u$.
By \Cref{lem:notAdditive}, because $K$ is outer additive,
\[\partial K = \Phi(a) \cup \Phi(b) = \Phi(a) \cup B_u.\]
By \Cref{clm:partialK},
\[K = \conv(\{a\} \cup B_u).\]
It follows that for every convex body $L \subset \mathcal{S}(K,u)$, we have $\lambda_{K}(L) = \lambda_{B_u}(L)$.
Hence, $B_u$ is outer additive in $\mathcal{S}(K,u) \cong \R^{n-1}$ because for every pair of convex bodies $L_1,L_2 \subset \mathcal{S}(K,u)$, \begin{align*}
\lambda_{B_u}(L_1+L_2) 
& = \lambda_{K}(L_1+L_2) \\
& = \lambda_{K}(L_1)+\lambda_{K}(L_2) \\
& = \lambda_{B_u}(L_1)+\lambda_{B_u}(L_2).
\end{align*}
By induction, $B_u$ is an $(n-1)$-dimensional simplex, so $K$ is an $n$-dimensional simplex.
\end{proof}

\begin{remark*}
The theorem is true even when the convex bodies $L_1,L_2$ in the definition of outer additivity
are restricted to be intervals. 
\end{remark*}

\section{Depth of polytopes}
\label{sec:depth}

\subsection{Depth of polytopes and neural networks}
\label{subsec:constructing-polytopes}

Consider the following model for constructing polytopes. The two operations are Minkowski sum 
($A+B = \{a+b: a\in A,b\in B\}$) and convex hull of a union
of two polytopes.
Polytopes of depth zero are points:
\[
\cP_{n,0} = \{ \{x\} : x \in \R^n\}.
\]
Polytopes of depth $d > 0$ are defined inductively as
\[
\cP_{n,d}
=
\left\{
\sum_{j=1}^m \conv(K_j\cup L_j)
:
m \in \N,\ K_j,L_j \in \cP_{n,d-1}
\right\}.
\]
For example, $\cP_{n,1}$ is the set of all zonotopes (sums of segments).
The depth complexity of a polytope $P \subset \R^n$ is the minimum $d$ so that $P \in \cP_{n,d}$.
Depth complexity is always finite. In fact, if $P$ has $m$ vertices then its depth is at most $\lceil \log_2 m \rceil$
(without using any Minkowski sums).

This model has a one-to-one correspondence with the following {\em input convex} neural network (ICNN) model for computing convex functions (defined in~\cite{amos2017input}).
The input gates compute linear functions on $\R^n$. Inner gates compute functions of the form $g(x) =  \sum_j c_j \max\{a_j(x),b_j(x)\}$
where for every $j$, it holds that $c_j > 0$ and $a_j,b_j$ were previously computed.
Motivation for studying ICNNs comes from the facts that ICNNs compute convex functions, and minimizing convex functions is a tractable goal. For example, ICNNs can be trained to represent a loss function that guides the training process of other models.

The correspondence between polytopes and convex functions is seen via support functions. 
Every closed convex set $P \subset \R^n$
defines a support function $h_P : \R^n \to \R$
via
\[h_P(x) = \max \{ \ip{x}{y} : y \in P \}.\]
Support functions have many important properties (see e.g.~\cite{gr2003unbaum}).
In particular, there is a one-to-one correspondence between  $P$ and $h_P$.
This correspondence translates to the computational setting as the following lemma shows. 
\begin{lemma*}[\cite{hertrich2021towards}]
For every polytope $X \in \R^n$, the required number of hidden layers to compute $h_X$ is equal to the minimum $d$ such that there are $P, Q \in \cP_{n,d}$ such that
\[X + P = Q.\]
\end{lemma*}

A particularly important polytope in this setting is the simplex.
The support function $h_{\Delta^n}$ corresponds to computing the maximum of $n+1$ numbers. 
The depth complexity of $\Delta^n$ is known to be $\ceil{\log_2(n+1)}$: the upper bound is immediate, and the matching lower bound was proved by Valerdi~\cite{valerdi2024minimal}.
This quantity is also important on the functional side, because if $\Delta^n$ has depth $d$, then every continuous piecewise linear (CPWL) function on $\R^{n-1}$ can be computed with $d$ hidden layers~\cite{arora2018understanding,wang2005generalization}.
In other words, $h_{\Delta^n}$ is complete for depth complexity with respect to all CPWL functions on $(n-1)$-dimensional space.
This led the authors of~\cite{hertrich2021towards} to conjecture that the number of hidden layers required to compute $h_{\Delta^n}$ is exactly $\lceil \log_2(n+1) \rceil$.
The authors of~\cite{bakaev2025better} constructed polytopes $P, Q \in \cP_{4,2}$ such that
\[
\Delta^4 + P = Q.
\]
This construction refuted the conjecture. With some additional work, ~\cite{bakaev2025better} deduced that
the number of hidden layers required to compute $h_{\Delta^n}$ is at most $\approx \log_3(n)$.
Other lower-bound results are known under additional restrictions, such as integral or rational weights~\cite{haase2023lower,averkov2025expressiveness}.

\subsection{Indecomposability and approximation}
\label{subsec:approx-polytopes}

There are numerous classical approximation problems in convex geometry, such as approximating a given convex set by polytopes and approximating a given polytope by specific polytopes. 
These problems, in their many forms, have been studied for many years with various types of motivation (see~\cite{boroczky2000approximation, gruber1983approximation,gr2003unbaum,kallay1982indecomposable} and references within).

Shephard and others considered the following approximation problem (see~\cite{shephard1964approximation} and references within).
A target convex set $T$ is approximated by a collection $\cK$ of convex sets
if there is a sequence of sets $(K_j)_{j=1}^\infty$ such that each $K_j$ is a finite Minkowski sum of elements of $\cK$ and $K_j \to T$ as $j \to \infty$
in the Hausdorff metric. 

As we describe next, Shephard proved that an indecomposable~$T$ can only be ``trivially approximated''. 
Recall that $\cH_K$ is the set of all homothets of $K$. 
A set $K \subset \R^n$ is indecomposable if it is ``prime with respect to Minkowski summation'':

\begin{defn}
    A set $K \subset \R^n$ is indecomposable if for every $L_1,\ldots,L_m \subset \R^n$ such that $K = L_1+\ldots+L_m$, it holds that $L_1,\ldots,L_m \in \cH_K$.
\end{defn}

Indecomposable polytopes were studied in many works (see e.g.~\cite{firey1964addition,shephard1964approximation,shephard1963decomposable}). 
It is known, e.g., that if $K$ is a polytope so that all of its two-dimensional faces are triangles then $K$ is indecomposable~\cite{shephard1963decomposable}.
Indecomposability was also recently used to study the expressivity of monotone ReLU neural networks and input convex neural networks (ICNNs);
see~\cite{bakaev2025depth,valerdi2024minimal}.
Returning to the approximation problem mentioned above, Shephard proved that an indecomposable $T$ can be approximated by a class of convex bodies $\cK$ only if $\cK$ contains a homothet of~$T$.

It is well-known that simplices are indecomposable.
One of our results is a stronger quantitative obstruction.
We show that if for some depth parameter $d$ there is a sequence of polytopes $(P_j)$ in $\cP_{n,d}$ such that $P_j \to \Delta^n$ as $j \to \infty$, then necessarily $d \geq \ceil{\log_2(n+1)}$.
In fact, we show that polytopes of depth complexity smaller than $\ceil{\log_2(n+1)}$ are {\em uniformly} bounded away from $\Delta^n$.

More generally,
we are interested in the depth complexity of {\em approximating} a given polytope up to some finite accuracy. 
We measure the quality of an approximation using a ``simplex-based'' measure.
As opposed to Shephard's result, in our non-asymptotic setting, some indecomposable polytopes can be approximated non-trivially.
The reason is that, in dimension $n \geq 3$, the set of indecomposable polytopes is dense~\cite{schneider2013convex}, so every polytope $P \subset \R^n$ is close to an indecomposable polytope $Q \subset \R^n$
(for example, we can build $Q$ that is close to $P$ such that all of its two-dimensional faces are triangles).

\subsection{The simplex-based measure and depth}
\label{sec:simplex-based-measure}

\begin{restatetheorem}{thm:Simplex}[restatement]
If a polytope $P$ has depth $d < \lceil \log_2(n+1) \rceil$, then 
\[\rho_\Delta(P) \le 2^d-1.\]
This bound is sharp (this inequality is an equality for some polytopes).
\end{restatetheorem}

\begin{proof}
The proof is by induction on $d$.
If $d=0$, then every element of $\cP_{n,0}$ is a point, so the claim is exactly the convention $\rho_\Delta(\{p\})=0$.
If $d=1$, every $P\in\cP_{n,1}$ is a zonotope and hence centrally symmetric, so $\rho_\Delta(P)=1$ and the statement is correct.
Now let $d>1$ and let
\[
P=\sum_j \conv(K_j \cup L_j) \in \cP_{n,d},
\qquad K_j,L_j\in\cP_{n,d-1}.
\]

We may assume that every summand $\conv(K_j \cup L_j)$ is not a point. 
We may also assume that every $K_j,L_j$ are non-singleton compact convex sets
(if, say, $K_j$ is a point, choose $x\in L_j\setminus K_j$ and replace $K_j$ by $K_j:=\conv(K_j \cup \{x\})$).
\Cref{thm:De} and \Cref{thm:DeUnion} imply
\begin{align*}
\rho_\Delta(P)
&\le \max_j \rho_\Delta\bigl(\conv(K_j \cup L_j)\bigr)\\
&\le \max_j \bigl(\rho_\Delta(K_j)+\rho_\Delta(L_j)+1\bigr)\\
&\le 2(2^{d-1}-1)+1\\
&= 2^d-1.
\end{align*}

This proves the inequality.
The inequality is achieved by construction. In depth $d$, we can build a $(2^d-1)$-dimensional simplex $\Delta^{(d)}$ with $2^d$ vertices that are vertices of $\Delta$, and then
\begin{equation*}
\rho_\Delta( \Delta^{(d)}) \ge 2^d-1.
\qedhere    
\end{equation*}

\end{proof}

\subsection{Depth complexity}
\label{sec:minkowski-depth}

The identification $\alpha(K)=m^*(K)$  turns the depth bounds for the simplex-based measure into depth bounds for the intrinsic Minkowski measure of symmetry. 

\begin{restatetheorem}{thm:minkowski-depth-upper}[restatement]
Let $P\subset\R^n$ be a non-singleton polytope, let $k=\dim P$, and suppose that
$P\in\cP_{n,d}$. If $0\le d<\lceil\log_2(k+1)\rceil$, then
\[
m^*(P)\le 2^d-1.
\]
\end{restatetheorem}

\begin{proof}
Fix $\Delta\in\mathscr S(P)$. 
 Let
$T:E_P\to\mathbb R^k$ be an invertible affine map with
$T(\Delta)=\Delta^k$, then $T(P)\in\mathcal P_{k,d}$ and
\[
\rho_\Delta(P)=\rho_{\Delta^k}(T(P)).
\]
\Cref{thm:Simplex} in dimension $k$ gives
\[
\rho_{\Delta^k}(T(P))\le 2^d-1.
\]
Taking the supremum over $\Delta\in\mathscr S(P)$ yields
\[
m^*(P)=\alpha(P)\le 2^d-1.
\qedhere 
\]
\end{proof}

\begin{remark*}
The bound in \Cref{thm:minkowski-depth-upper} 
is sharp 
as in \Cref{thm:Simplex}.
\end{remark*}

\subsection{Inapproximability}
\label{sec:inapproximability}

\Cref{thm:Simplex} also allows one to move from inapproximability on the ``geometric side''
to inapproximability on the ``function side''.
On the geometric side, the distance of a polytope $P$ from the simplex is captured by the simplex-based measure $\rho_\Delta(P)$
where $\Delta = \Delta^n$.
On the function side, we can measure the distance of $h_P$ from $h_\Delta$ by sampling $X \sim \S^{n-1}$ from the uniform (Haar) probability measure on the sphere via
 \[\E |h_P(X) - h_{\Delta^n}(X)| .\]
The Blaschke selection theorem 
says that the collection of compact convex sets is locally compact. It implies that for every $n$ and $\delta > 0$,
there is $\eps>0$ such that if $\rho_\Delta(P) \leq n-\delta$ then
 \[\E |h_P(X) - h_{\Delta}(X)| \geq \eps.\]
By \Cref{thm:Simplex}, there is $\eps_0 = \eps_0(n) > 0$ such that if the polytope $P$ has depth $d < \lceil \log_2 (n+1) \rceil$ then 
 \[\E |h_P(X) - h_{\Delta}(X)| \geq \eps_0.\]
In words, support functions of polytopes of small depth are far from $h_\Delta$
(a.k.a.\ the $\max_{n+1}$ function).
It is worth noting that $\eps_0 \to 0$ as $n \to \infty$ because the simplex is approximated by a point (under this definition).
Indeed, $\E |h_{\Delta}(X)| \leq O(\tfrac{\log n}{\sqrt{n}})$ 
so for $P = \{0\} \in \cP_{n,0}$,
\[\E |h_P(X) - h_{\Delta}(X)| \to 0 \]
as $n \to \infty$.

\newcommand{\etalchar}[1]{$^{#1}$}


\begin{thebibliography}{HBDSS21}

\bibitem[ABMM18]{arora2018understanding}
Raman Arora, Amitabh Basu, Poorya Mianjy, and Anirbit Mukherjee.
\newblock Understanding deep neural networks with rectified linear units.
\newblock In {\em International Conference on Learning Representations}, 2018.

\bibitem[AHM25]{averkov2025expressiveness}
Gennadiy Averkov, Christopher Hojny, and Maximilian Merkert.
\newblock On the expressiveness of rational {ReLU} neural networks with bounded
  depth.
\newblock In {\em 13th international Conference on Learning Representations,
  ICLR 2025}, 2025.

\bibitem[AK12]{akopyan2012kadets}
Arseniy Akopyan and Roman Karasev.
\newblock Kadets-type theorems for partitions of a convex body.
\newblock {\em Discrete \& Computational Geometry}, 48:766--776, 2012.

\bibitem[AXK17]{amos2017input}
Brandon Amos, Lei Xu, and J.~Zico Kolter.
\newblock Input convex neural networks.
\newblock In {\em International conference on machine learning}, pages
  146--155. PMLR, 2017.

\bibitem[BBH{\etalchar{+}}25]{bakaev2025depth}
Egor Bakaev, Florestan Brunck, Christoph Hertrich, Daniel Reichman, and Amir
  Yehudayoff.
\newblock On the depth of monotone {ReLU} neural networks and {ICNNs}.
\newblock {\em arXiv preprint arXiv:2505.06169}, 2025.

\bibitem[BBH{\etalchar{+}}26]{bakaev2025better}
Egor Bakaev, Florestan Brunck, Christoph Hertrich, Jack Stade, and Amir
  Yehudayoff.
\newblock Better neural network expressivity: Subdividing the simplex.
\newblock In {\em Proceedings of the 58th Annual ACM Symposium on Theory of
  Computing}, STOC '26, page 500–507, 2026.

\bibitem[BJ96]{boroczky1996around}
K{\'a}roly B{\"o}r{\"o}czky~Jr.
\newblock Around the {Rogers--Shephard} inequality.
\newblock {\em Mathematica Pannonica}, (7):113--130, 1996.

\bibitem[BJ00]{boroczky2000approximation}
K{\'a}roly B{\"o}r{\"o}czky~Jr.
\newblock Approximation of general smooth convex bodies.
\newblock {\em Advances in Mathematics}, 153(2):325--341, 2000.

\bibitem[BJ05]{boroczky2005stability}
K{\'a}roly B{\"o}r{\"o}czky~Jr.
\newblock The stability of the {Rogers--Shephard} inequality and some related
  inequalities.
\newblock {\em Advances in Mathematics}, 190(1):47--76, 2005.

\bibitem[FG64]{firey1964addition}
William~J. Firey and Branko Gr{\"u}nbaum.
\newblock Addition and decomposition of convex polytopes.
\newblock {\em Israel Journal of Mathematics}, 2(2):91--100, 1964.

\bibitem[Gr{\"u}63]{grunbaum1963measures}
Branko Gr{\"u}nbaum.
\newblock Measures of symmetry for convex sets.
\newblock In {\em Convexity: Proceedings of the Seventh Symposium in Pure
  Mathematics of the American Mathematical Society}, volume~7, page 233.
  American Mathematical Soc., 1963.

\bibitem[Gru83]{gruber1983approximation}
Peter~M. Gruber.
\newblock Approximation of convex bodies.
\newblock In {\em Convexity and its Applications}, pages 131--162. Springer,
  1983.

\bibitem[Gr{\"u}03]{gr2003unbaum}
Branko Gr{\"u}nbaum.
\newblock {\em Convex Polytopes}.
\newblock Springer-Verlag, New York, 2003.

\bibitem[Guo05]{guo2005stability}
Qi~Guo.
\newblock Stability of the {Minkowski} measure of asymmetry for convex bodies.
\newblock {\em Discrete \& Computational Geometry}, 34(2):351--362, 2005.

\bibitem[HBDSS21]{hertrich2021towards}
Christoph Hertrich, Amitabh Basu, Marco Di~Summa, and Martin Skutella.
\newblock Towards lower bounds on the depth of {ReLU} neural networks.
\newblock {\em Advances in Neural Information Processing Systems},
  34:3336--3348, 2021.

\bibitem[HHL23]{haase2023lower}
Christian~Alexander Haase, Christoph Hertrich, and Georg Loho.
\newblock Lower bounds on the depth of integral {ReLU} neural networks via
  lattice polytopes.
\newblock In {\em The Eleventh International Conference on Learning
  Representations}, 2023.

\bibitem[Kal82]{kallay1982indecomposable}
Michael Kallay.
\newblock Indecomposable polytopes.
\newblock {\em Israel Journal of Mathematics}, 41:235--243, 1982.

\bibitem[Kla10]{klain2010containment}
Daniel~A. Klain.
\newblock Containment and inscribed simplices.
\newblock {\em Indiana University Mathematics Journal}, pages 1231--1244, 2010.

\bibitem[Sch09]{schneider2009stability}
Rolf Schneider.
\newblock Stability for some extremal properties of the simplex.
\newblock {\em J. Geom}, 96(1):135--148, 2009.

\bibitem[Sch13]{schneider2013convex}
Rolf Schneider.
\newblock {\em Convex bodies: the {Brunn--Minkowski} theory}, volume 151.
\newblock Cambridge university press, 2013.

\bibitem[She63]{shephard1963decomposable}
Geoffrey~C. Shephard.
\newblock Decomposable convex polyhedra.
\newblock {\em Mathematika}, 10(2):89--95, 1963.

\bibitem[She64]{shephard1964approximation}
Geoffrey~C. Shephard.
\newblock Approximation problems for convex polyhedra.
\newblock {\em Mathematika}, 11(1):9–18, 1964.

\bibitem[Tot13]{toth2013notes}
Gabor Toth.
\newblock Notes on {Schneider}’s stability estimates for convex sets.
\newblock {\em Journal of Geometry}, 104(3):585--598, 2013.

\bibitem[Tot15]{toth2015measures}
Gabor Toth.
\newblock {\em Measures of symmetry for convex sets and stability}, volume 776.
\newblock Springer, 2015.

\bibitem[Val24]{valerdi2024minimal}
Juan~L. Valerdi.
\newblock On minimal depth in neural networks.
\newblock {\em arXiv:2402.15315}, 2024.

\bibitem[WS05]{wang2005generalization}
Shuning Wang and Xusheng Sun.
\newblock Generalization of hinging hyperplanes.
\newblock {\em IEEE Transactions on Information Theory}, 51(12):4425--4431,
  2005.

\end{thebibliography}
\end{document}